\documentclass[12pt]{amsart}

\usepackage{amsmath,amssymb,amsbsy,amsfonts,amsthm,latexsym,
                     amsopn,amstext,amsxtra,euscript,amscd,mathrsfs}
\usepackage{amsmath,amssymb,amsbsy,amsfonts,latexsym,amsopn,amstext,cite,
                                               amsxtra,euscript,amscd,bm}
\usepackage{mathtools}
\usepackage{todonotes}
\usepackage{url}
\usepackage[colorlinks,linkcolor=blue,anchorcolor=blue,citecolor=blue,backref=page]{hyperref}
\usepackage{multirow}
\usepackage{float}
\usepackage{xcolor}

\usepackage[np]{numprint}
\npdecimalsign{\ensuremath{.}}

\usepackage{bibentry}

\DeclareMathOperator{\mx}{max}
\DeclareMathOperator{\id}{Id}

\usepackage[norefs,nocites]{refcheck}

\renewcommand*{\backref}[1]{}
\renewcommand*{\backrefalt}[4]{%
    \ifcase #1 (Not cited.)%
    \or        (p.\,#2)%
    \else      (pp.\,#2)%
    \fi}
\begin{document}

%%%%%%%%%%%%%%%%%%%%%%%%%%%%%%%%%%%%%%%%%%%%%%%%%%%%%%%%
%%%%%%%%%%%%%%%%%%%%%%%%%%%%%%%%%%%%%%%%%%%%%%%%%%%%%%%%
%%%%%%%%%%%%%%%%%%%%%%%%%%%%%%%%%%%%%%%%%%%%%%%%%%%%%%%%
%%%%%%%%%%%%%%%%%%%%%%%%%%%%%%%%%%%%%%%%%%%%%%%%%%%%%%%%

%%%%%%%   STANDARD STUFF %%%%%%%%%

%%%%%%%%%%%%%%%%%%%%%%%%%%%%%%%%%%%%%%%%%%%%%%%%%%%%%%%%
%%%%%%%%%%%%%%%%%%%%%%%%%%%%%%%%%%%%%%%%%%%%%%%%%%%%%%%%
%%%%%%%%%%%%%%%%%%%%%%%%%%%%%%%%%%%%%%%%%%%%%%%%%%%%%%%%
%%%%%%%%%%%%%%%%%%%%%%%%%%%%%%%%%%%%%%%%%%%%%%%%%%%%%%%%

%  use the AMS-Euler Fraktur fonts
%%%%%%%%%%%%%%%%%%%%%%%%%%%%%%%%%%
\newfont{\teneufm}{eufm10}
\newfont{\seveneufm}{eufm7}
\newfont{\fiveeufm}{eufm5}
%%%%%%%%%%%%%%%%%%%%%%%%%%%%%%%%%
%
%  allow automatic size selection in math mode
%
%%%%%%%%%%%%%%%%%%%%%%%%%%%%%%%%%
\newfam\eufmfam
 \textfont\eufmfam=\teneufm \scriptfont\eufmfam=\seveneufm
 \scriptscriptfont\eufmfam=\fiveeufm
%%%%%%%%%%%%%%%%%%%%%%%%%%%%%%%%%
%
%  \frak works on a single symbol at a time...
%
\def\frak#1{{\fam\eufmfam\relax#1}}
%

%%%%%%%%%%%%%%%%%%%  bbb-matter

\def\bbbr{{\rm I\!R}} %reelle Zahlen
\def\bbbm{{\rm I\!M}}
\def\bbbn{{\rm I\!N}} %natuerliche Zahlen
\def\bbbf{{\rm I\!F}}
\def\bbbh{{\rm I\!H}}
\def\bbbk{{\rm I\!K}}
\def\bbbp{{\rm I\!P}}
\def\bbbone{{\mathchoice {\rm 1\mskip-4mu l} {\rm 1\mskip-4mu l}
{\rm 1\mskip-4.5mu l} {\rm 1\mskip-5mu l}}}
\def\bbbc{{\mathchoice {\setbox0=\hbox{$\displaystyle\rm C$}\hbox{\hbox
to0pt{\kern0.4\wd0\vrule height0.9\ht0\hss}\box0}}
{\setbox0=\hbox{$\textstyle\rm C$}\hbox{\hbox
to0pt{\kern0.4\wd0\vrule height0.9\ht0\hss}\box0}}
{\setbox0=\hbox{$\scriptstyle\rm C$}\hbox{\hbox
to0pt{\kern0.4\wd0\vrule height0.9\ht0\hss}\box0}}
{\setbox0=\hbox{$\scriptscriptstyle\rm C$}\hbox{\hbox
to0pt{\kern0.4\wd0\vrule height0.9\ht0\hss}\box0}}}}
\def\bbbq{{\mathchoice {\setbox0=\hbox{$\displaystyle\rm
Q$}\hbox{\raise
0.15\ht0\hbox to0pt{\kern0.4\wd0\vrule height0.8\ht0\hss}\box0}}
{\setbox0=\hbox{$\textstyle\rm Q$}\hbox{\raise
0.15\ht0\hbox to0pt{\kern0.4\wd0\vrule height0.8\ht0\hss}\box0}}
{\setbox0=\hbox{$\scriptstyle\rm Q$}\hbox{\raise
0.15\ht0\hbox to0pt{\kern0.4\wd0\vrule height0.7\ht0\hss}\box0}}
{\setbox0=\hbox{$\scriptscriptstyle\rm Q$}\hbox{\raise
0.15\ht0\hbox to0pt{\kern0.4\wd0\vrule height0.7\ht0\hss}\box0}}}}
\def\bbbt{{\mathchoice {\setbox0=\hbox{$\displaystyle\rm
T$}\hbox{\hbox to0pt{\kern0.3\wd0\vrule height0.9\ht0\hss}\box0}}
{\setbox0=\hbox{$\textstyle\rm T$}\hbox{\hbox
to0pt{\kern0.3\wd0\vrule height0.9\ht0\hss}\box0}}
{\setbox0=\hbox{$\scriptstyle\rm T$}\hbox{\hbox
to0pt{\kern0.3\wd0\vrule height0.9\ht0\hss}\box0}}
{\setbox0=\hbox{$\scriptscriptstyle\rm T$}\hbox{\hbox
to0pt{\kern0.3\wd0\vrule height0.9\ht0\hss}\box0}}}}
\def\bbbs{{\mathchoice
{\setbox0=\hbox{$\displaystyle     \rm S$}\hbox{\raise0.5\ht0\hbox
to0pt{\kern0.35\wd0\vrule height0.45\ht0\hss}\hbox
to0pt{\kern0.55\wd0\vrule height0.5\ht0\hss}\box0}}
{\setbox0=\hbox{$\textstyle        \rm S$}\hbox{\raise0.5\ht0\hbox
to0pt{\kern0.35\wd0\vrule height0.45\ht0\hss}\hbox
to0pt{\kern0.55\wd0\vrule height0.5\ht0\hss}\box0}}
{\setbox0=\hbox{$\scriptstyle      \rm S$}\hbox{\raise0.5\ht0\hbox
to0pt{\kern0.35\wd0\vrule height0.45\ht0\hss}\raise0.05\ht0\hbox
to0pt{\kern0.5\wd0\vrule height0.45\ht0\hss}\box0}}
{\setbox0=\hbox{$\scriptscriptstyle\rm S$}\hbox{\raise0.5\ht0\hbox
to0pt{\kern0.4\wd0\vrule height0.45\ht0\hss}\raise0.05\ht0\hbox
to0pt{\kern0.55\wd0\vrule height0.45\ht0\hss}\box0}}}}
\def\bbbz{{\mathchoice {\hbox{$\sf\textstyle Z\kern-0.4em Z$}}
{\hbox{$\sf\textstyle Z\kern-0.4em Z$}}
{\hbox{$\sf\scriptstyle Z\kern-0.3em Z$}}
{\hbox{$\sf\scriptscriptstyle Z\kern-0.2em Z$}}}}
\def\ts{\thinspace}

\newtheorem{thm}{Theorem}
\newtheorem{lem}{Lemma}
\newtheorem{lemma}[thm]{Lemma}
\newtheorem{prop}{Proposition}
\newtheorem{proposition}[thm]{Proposition}
\newtheorem{theorem}[thm]{Theorem}
\newtheorem{cor}[thm]{Corollary}
\newtheorem{corollary}[thm]{Corollary}
\newtheorem{rem}[thm]{Remark}
\newtheorem{conj}[thm]{Conjecture}
\newtheorem{exe}[thm]{Example}

\newtheorem{prob}{Problem}
\newtheorem{problem}[prob]{Problem}
\newtheorem{quest}[thm]{Question}
\newtheorem{question}[thm]{Question}

%% DEFINITIONS

\numberwithin{equation}{section}
\numberwithin{thm}{section}
\numberwithin{table}{section}

\def\squareforqed{\hbox{\rlap{$\sqcap$}$\sqcup$}}
\def\qed{\ifmmode\squareforqed\else{\unskip\nobreak\hfil
\penalty50\hskip1em\null\nobreak\hfil\squareforqed
\parfillskip=0pt\finalhyphendemerits=0\endgraf}\fi}

\def\fF{\EuScript{F}}

\def\fJ{\mathfrak{J}}

%%%%%%%%%%%%%%%%%%%%%%%%%
% Alphabet calligraphie %
%%%%%%%%%%%%%%%%%%%%%%%%%
\def\cA{{\mathcal A}}
\def\cB{{\mathcal B}}
\def\cC{{\mathcal C}}
\def\cD{{\mathcal D}}
\def\cE{{\mathcal E}}
\def\cF{{\mathcal F}}
\def\cG{{\mathcal G}}
\def\cH{{\mathcal H}}
\def\cI{{\mathcal I}}
\def\cJ{{\mathcal J}}
\def\cK{{\mathcal K}}
\def\cL{{\mathcal L}}
\def\cM{{\mathcal M}}
\def\cN{{\mathcal N}}
\def\cO{{\mathcal O}}
\def\cP{{\mathcal P}}
\def\cQ{{\mathcal Q}}
\def\cR{{\mathcal R}}
\def\cS{{\mathcal S}}
\def\cT{{\mathcal T}}
\def\cU{{\mathcal U}}
\def\cV{{\mathcal V}}
\def\cW{{\mathcal W}}
\def\cX{{\mathcal X}}
\def\cY{{\mathcal Y}}
\def\cZ{{\mathcal Z}}
\def\pgdc{\textrm{gcd}}
\newcommand{\rmod}[1]{\: \mbox{mod} \: #1}

\def\Nm{{\mathrm{Nm}}}

\def\Tr{{\mathrm{Tr}}}

\def\epp{\mathbf{e}_{p-1}}

\def\ind{\mathop{\mathrm{ind}}}

\def\mand{\qquad \mbox{and} \qquad}

\def\M{\mathsf {M}}
\def\T{\mathsf {T}}

\newcommand{\commI}[1]{\marginpar{%
\begin{color}{magenta}
\vskip-\baselineskip %raise the marginpar a bit
\raggedright\footnotesize
\itshape\hrule \smallskip I: #1\par\smallskip\hrule\end{color}}}

\newcommand{\commO}[1]{\marginpar{%
\begin{color}{blue}
\vskip-\baselineskip %raise the marginpar a bit
\raggedright\footnotesize
\itshape\hrule \smallskip O: #1\par\smallskip\hrule\end{color}}}

\newcommand{\commM}[1]{\marginpar{%
\begin{color}{blue}
\vskip-\baselineskip %raise the marginpar a bit
\raggedright\footnotesize
\itshape\hrule \smallskip M: #1\par\smallskip\hrule\end{color}}}

%%%%%%%%%%%%%%%%%%%%%%%%%%%%%%%%%%%%%%%%%%%%%%%%%%%%%%%%
%%%%%%%%%%%%%%%%%%%%%%%%%%%%%%%%%%%%%%%%%%%%%%%%%%%%%%%%
%%%%%%%%%%%%%%%%%%%%%%%%%%%%%%%%%%%%%%%%%%%%%%%%%%%%%%%%
%%%%%%%%%%%%%%%%%%%%%%%%%%%%%%%%%%%%%%%%%%%%%%%%%%%%%%%%

%%%%%%%  END OF STANDARD STUFF %%%%%%%%%

%%%%%%%%%%%%%%%%%%%%%%%%%%%%%%%%%%%%%%%%%%%%%%%%%%%%%%%%
%%%%%%%%%%%%%%%%%%%%%%%%%%%%%%%%%%%%%%%%%%%%%%%%%%%%%%%%
%%%%%%%%%%%%%%%%%%%%%%%%%%%%%%%%%%%%%%%%%%%%%%%%%%%%%%%%
%%%%%%%%%%%%%%%%%%%%%%%%%%%%%%%%%%%%%%%%%%%%%%%%%%%%%%%
\newcommand{\ignore}[1]{}

\hyphenation{re-pub-lished}

\parskip 1.5 mm

\def\GL{\operatorname{GL}}
\def\SL{\operatorname{SL}}
\def\PGL{\operatorname{PGL}}
\def\PSL{\operatorname{PSL}}
\def\li{\operatorname{li}}

\def\vec#1{\mathbf{#1}}

\def \F{{\mathbb F}}
\def \K{{\mathbb K}}
\def \Z{{\mathbb Z}}
\def \N{{\mathbb N}}
\def \Q{{\mathbb Q}}
\def \C {{\mathbb C}}
\def \R{{\mathbb R}}
\def\Fp{\F_p}
\def \fp{\Fp^*}

\def \Rc{{\mathcal R}}
\def \Qc{{\mathcal Q}}
\def \Ec{{\mathcal E}}

\def \DN{D_N}
\def\va{\mbox{\bf{a}}}

\def\Kc{\,{\mathcal K}}
\def\Ic{\,{\mathcal I}}

\def\\{\cr}
\def\({\left(}
\def\){\right)}
\def\fl#1{\left\lfloor#1\right\rfloor}
\def\rf#1{\left\lceil#1\right\rceil}

\def\Ln#1{\mbox{\rm {Ln}}\,#1}

\def \nd {\, | \hspace{-1.2mm}/\,}

 \def\e{\mathbf{e}}

\def\ep{\mathbf{e}_p}
\def\eq{\mathbf{e}_q}

\def\wt#1{\mbox{\rm {wt}}\,#1}

\def\Mob{M{\"o}bius }

%%%%%%%%%%%%%%%  Topmatter %%%%%%%%%%%%%%%%%%

\title{On a sum involving the Euler function}

\author[O. Bordell\`{e}s]{Olivier Bordell\`{e}s}
\address{O.B.: 2 All\'ee de la combe, 43000 Aiguilhe, France}
\email{borde43@wanadoo.fr}

\author[L. Dai]{Lixia Dai}
\address{L.D.: School of Mathematical Sciences, Nanjing Normal University, Nanjing 210046, People's Republic of China}
\email{lilidainjnu@163.com}

\author[R. Heyman] {Randell~Heyman}
\address{R.H.: Department of Pure Mathematics, University of New South Wales,
Sydney, NSW 2052, Australia.}
\email{randell@unsw.edu.au}

\author[H. Pan] {Hao Pan}
\address{H.P.: School of Applied Mathematics, Nanjing University of Finance and Economics, Nanjing 210046, People's Republic of China}
\email{haopan79@zoho.com}

\author[I. E.  Shparlinski]{Igor E. Shparlinski}
\address{I.S.: Department of Pure Mathematics, University of New South Wales\\
2052 NSW, Australia.}
\email{igor.shparlinski@unsw.edu.au}

%\date{\today}

\date{}

\pagenumbering{arabic}

\begin{abstract}
We obtain reasonably tight  upper and lower bounds on the sum $\sum_{n \leqslant x} \varphi \( \fl{x/n} \)$, involving the Euler functions $\varphi$ and the integer parts $\fl{x/n} $ of the reciprocals of integers.
\end{abstract}


\maketitle

%\section{Introduction}

\section{Background and motivation}

Let, as usual, for an integer $n \geqslant1$,
$\varphi(n)$ denote the Euler function, that is,  the number of units in
the residue ring $\Z/n\Z$.

By a classical result of Walfisz~\cite{Wal}, we have the
following asymptotic formula for the summary function of the Euler function
$$\sum_{n \leqslant x}\varphi(n)=\frac{x^2}{2 \zeta(2)}+
O\(x(\log x)^{2/3}(\log\log x)^{4/3}\),$$
see also~\cite[Theorem~6.44]{Bord1}.

Furthermore, for any real number $x$ we denote by $\fl{x}$
its integer part, that is, the greatest integer that does not exceed $x$.
The most straightforward sum of the floor function is related to the divisor summatory function since
$$\sum_{n \leqslant x} \fl{x/n}=\sum_{n \leqslant x}\sum_{k \leqslant x/n}1=\sum_{n \leqslant x}\tau(n),$$
where $\tau(n)$ is the number of divisors of $n$. From \cite[Theorem~2]{BouWat} we infer
$$\sum_{n \leqslant x} \fl{x/n}= x \log x + x (2 \gamma - 1) + O\( x^{517/1648+ o(1)}\),$$
where $\gamma$ is  the \textit{Euler--Mascheroni} constant, in particular $\gamma  \approx 0.57\, 722$.
% It should be pointed out that some  sums with
%fractional parts $\{x/n\}$ have recently been considered by Shubin~\cite{Shu}.

Here we combine both functions and consider an apparently new type
of sums, namely,
$$
S(x)=\sum_{n \leqslant x} \varphi\(\fl{x/n}\).
$$
The sum $S(x)$ is also a mean value of a certain divisor function, as it may be seen by interchanging the summations. More precisely, if $\tau_x$ is the divisor function defined by
$$\tau_x(n) = \sum_{\substack{d \mid n \\ \gcd\(d, \left \lfloor dx/n \right \rfloor \right) =1}} 1$$
then
$$S(x) = \sum_{n \leqslant x} \tau_x (n).$$
Note that, for each fixed real number $x \geqslant 1$, the arithmetic function $\tau_x$ is not multiplicative, which explains why an asymptotic formula for $S(x)$ is quite difficult to get. However, the aim of this work is to obtain reasonably tight upper and lower bounds for this sum.

We also consider more general sums of arithmetic functions with $\fl{x/n}$, and in the case
of functions growing slower than the Euler function we obtain asymptotic formulas
for such sums.

We remark our work is partially motivated by the extensive body of research on
arithmetic functions with integer parts of real-valued functions, most commonly. with
 Beatty $\fl{\alpha n + \beta}$  sequences, see, for
 example,~\cite{ABS,BaBa,BaLi,GuNe,Harm}, and
 Piatetski--Shapiro $\fl{n^\gamma}$ sequences,
 see, for example,~\cite{Akb,BBBSW,BBGY,BGS,LSZ,Morg},
 with real $\alpha$, $\beta$ and $\gamma$. In particular, we obtain an
 analogue of the result of  Morgenbesser~\cite{Morg} on the sum of
 digits of $\fl{n^c}$ for the sequence $\fl{x/n}$, see Example~\ref{exe:digit}
 below.

\section{Main Results}

\subsection{The Euler function}
We start with upper and lower bounds on $S(x)$.

\begin{theorem}
\label{thm:Sx}
Uniformly, for all $x \geqslant 3$,
\begin{align*}
\(\frac{2629}{4009}\cdot\frac1{\zeta(2)} + \frac{1380}{4009}+o(1)\)& x\log x\\
   \geqslant  S(x) & \geqslant \( \frac{ {2629}}{ {4009}} \cdot\frac1{\zeta(2)} + o(1) \) x \log x ,
\end{align*}
as $x \to \infty$.
\end{theorem}
 
The proofs of both  lower and upper bounds of Theorem~\ref{thm:Sx}
relying on the theory of exponent pairs, see~\cite[Chapter~6]{Bord1}.
In particular, to obtain the numerically strongest result,  we use the
recently discovered exponent pair of Bourgain~\cite{Bour} combined
with so called $A$- and $B$-processes,
see~\cite[Sections~6.4.2 and~6.6.2]{Bord1}.
We remark that in the  lower of Theorem~\ref{thm:Sx}
 the quantity in $o(1)$ is negative.

We note it is natural to ask the following:

\begin{quest}
\label{quest:asymp}
Is it true that
$$
S(x) = \(\frac{1}{\zeta(2)} +o(1)\)x \log x
$$
as  $x \to \infty$?
\end{quest}

In Section~\ref{sec:comp} we present some numerical data which makes us rather cautiously believe that the answer to Question~\ref{quest:asymp} is positive.

\subsection{Slowly growing arithmetic functions}
One of the difficulties in investigating the sum $S(x)$ is a large size of $\varphi(n)$. In particular, some individual terms of the sum $S(x)$ are only logarithmically smaller than the entire sum. However,
for slowly growing  arithmetic functions $f(n)$ in similar sums,
$$
S_f(x) = \sum_{n \leqslant x} f \(\fl{x/n} \),
$$
we are able to get an asymptotic formula.

Let $\tau_k(n)$ denotes  the  generalised divisor function, which is defined as the number
of ordered representations $n = d_1\ldots d_k$ with  integer numbers $d_1, \ldots, d_k \geqslant1$.
In particular $\tau_1(n)=1$.

We also define $  \varepsilon_1(x) = 0$ and
\begin{equation}
   \varepsilon_k(x) = \sqrt{\frac{k \log \log \log x}{\log \log x}} \( k -1 + \frac{30}{\log \log \log x} \), \label{eq:eps}
\end{equation}
for $k \geqslant 2$.
Now we have the obvious estimate $\varepsilon_k(x) = o(1)$ as $x \to \infty$.

We write $O_k$ to indicate that in the relations $U = O_k(V)$ the implied
constant may depend on $k$.
We also write $\Z_{\geqslant k} $ for the set
$$
\Z_{\geqslant k} = \Z\cap[k, \infty).
$$

\begin{theorem}
\label{thm:Sfx}
Let $f$ be a complex-valued arithmetic function such that there exist $A > 0$ and $k \in \Z_{\geqslant 1}$ such that $\left | f \right | \leqslant A \tau_k$. Then
$$\sum_{n \leqslant x} f \(\fl{x/n} \) = x \sum_{n=1}^{\infty} \frac{f(n)}{n(n+1)} + O_{k} \( Ax^{1/2} (\log x)^{\delta_k+\varepsilon_{k+1}(x)/2}\),$$
where $\delta_1 = 0$ if $k=1$, $\delta_k = k- 1/2$ if $k \geqslant 2$,
and where $\varepsilon_{k+1}(x)$ is defined in~\eqref{eq:eps}.
\end{theorem}

In particular, applying Theorem~\ref{thm:Sfx} to $f(n) = \varphi(n)/n$ (and using $k=1$) we obtain:

\begin{cor}
\label{cor:phi-asymp}
We have
$$\sum_{n \leqslant x} \frac{\varphi \(\fl{x/n} \)}{\fl{x/n}} = \kappa x  + O \(x^{1/2}\),$$
where
$$
\kappa =  \sum_{n=1}^{\infty} \frac{\varphi(n)}{n^2(n+1)} \approx 0.78 \, 838.
$$
\end{cor}

Finally, the method of proof of Theorem~\ref{thm:Sfx} can be extended to more general and faster growing arithmetic functions   at the cost of a weaker error term.

We use
$$
\phi = \frac{1 + \sqrt{5}}{2} \approx 1.61 \, 803
$$
to denote  the Golden ratio.

\begin{theorem}
\label{thm:Sfxinf}
Let $f$ be a complex-valued arithmetic function and assume that there exists $A > 0$ such that $$
\left | f(n) \right | \ll n^{\phi - 1} (\log en)^{-A}.
$$
Then
$$\sum_{n \leqslant x} f \(\fl{x/n} \) = x \sum_{n=1}^{\infty} \frac{f(n)}{n(n+1)} + O \( x (\log x)^{-A(\phi-1)} \).$$
\end{theorem}

We   also have a result which depends on the average  behaviour of arithmetic functions,
which is very useful for functions with irregular behaviour. We give several examples of such
functions in Section~\ref{sec:examp}

\begin{theorem}
\label{thm:Sfx2}
Let $f$ be a complex-valued arithmetic function and assume that there exists $0 < \alpha < 2$ such that
\begin{equation}
   \sum_{n \leqslant x} \left | f(n) \right |^2 \ll x^{\alpha}. \label{eq:hyp-2.5}
\end{equation}
Then
$$\sum_{n \leqslant x} f \(\fl{x/n} \) = x \sum_{n=1}^{\infty} \frac{f(n)}{n(n+1)} + O \( x^{ (\alpha +1)/3} (\log x)^{(1+\alpha) \(2+\varepsilon_2(x) \)/6} \)$$
where $\varepsilon_2(x)$ is given in~\eqref{eq:eps}.
\end{theorem}

In particular, if in Theorem~\ref{thm:Sx} one replaces the sum $S(x)$ with $\varphi(n)$ with
a similar sum with  $\varphi(n)^{\beta}$ for some $\beta< 1/2$, then Theorem~\ref{thm:Sfx2}
immediately applies and implies an asymptotic formula.

\section{Some applications}
\label{sec:examp}

Here we give some examples of interesting arithmetic functions to which we can apply our results.

\begin{exe}
\label{ex:Omega-3}
Let $f(n) = \sqrt{3}^{\; \Omega(n)}$. From~\cite[Chapter~I.3, Exercise~58(f)]{Ten}, we have
$$\sum_{n \leqslant x} f(n)^2 \ll x^{\log 3/\log 2}$$
(indeed one verifies that the functions $f(\vartheta)$ and $g(\vartheta)$ involved in the
asymptotic formulas of~\cite[Chapter~I.3, Exercise~58(f)]{Ten} are both monotonically decreasing).
Hence,  Theorem~\ref{thm:Sfx2} gives
$$\sum_{n \leqslant x} \sqrt{3}^{\; \Omega \(\fl{x/n} \) } = x \sum_{n=1}^{\infty} \frac{\sqrt{3}^{\; \Omega(n)}}{n(n+1)} + O \( x^{\log 6/\log 8} (\log x)^{\log 6/\log 8 + o(1)}\).$$
Note that
$$
 \sum_{n=1}^{\infty} \frac{\sqrt{3}^{\; \Omega(n)}}{n(n+1)} \approx 1.77 \, 694.
$$
\end{exe}

Clearly in Example~\ref{ex:Omega-3} one can take $f(n) = \lambda^{\Omega(n)}$
with any $\lambda \leqslant \sqrt{3}$ and still have an asymptotic formula.
We now show that one can also take a slightly larger values of $\lambda$.

\begin{exe}
Let $f(n) = \lambda^{\Omega(n)}$ with $\sqrt{3}  \leqslant \lambda < 2$.
Combining  the trivial bound
$$
\Omega(n)  \leqslant \frac{\log n}{\log 2}
$$
with~\cite[Chapter~I.3, Exercise~58(f)]{Ten}, we derive
\begin{align*}
\sum_{n \leqslant x} f(n)^2
&= \sum_{n \leqslant x} 3^{\Omega(n)} (\lambda^2/3)^{\Omega(n)}\\
& \leqslant x^{\log (\lambda^2/3)/\log 2} \sum_{n \leqslant x} 3^{\Omega(n)}
\ll x^{2\log \lambda/\log 2}.
\end{align*}
Hence by  Theorem~\ref{thm:Sfx2}, for  any positive  $ \lambda < 2$ there exists
some $\kappa>0$ such that
$$\sum_{n \leqslant x} \lambda^{\Omega\(\fl{x/n} \) } = x \sum_{n=1}^{\infty} \frac{\lambda^{\Omega(n)}}{n(n+1)} + O \( x^{1-\kappa}\).$$
\end{exe}

We now give an application of Theorem~\ref{thm:Sfx2} to a very different function.

\begin{exe}
Let $k \in \Z_{\geqslant 2}$ and define $M_k(n)$ to be the maximal $k$-full divisor of $n$
(see~\cite{sur}). Since
$$L \left( s,M_k^2 \right) = \zeta(s) \prod_p \left( 1  + \frac{p^{2k}-1}{p^{ks}} \right)  \quad \left( \sigma > 2 + \tfrac{1}{k} \right)$$
we infer that
$$\sum_{n \leqslant x} M_k(n)^2 \leqslant x^{2 + 1/k +o_k(1)},$$
where $o_k(1)$ denotes a quantity which for a fixed $k$ tends to zero as $x \to \infty$.
Now let $f_k(n) = n^{-1/k} M_k(n)$. By partial summation, we obtain from the above estimate
$$\sum_{n \leqslant x} f_k(n)^2 \leqslant x^{2 - 1/k +o_k(1)}.$$
Applying Theorem~\ref{thm:Sfx2} we derive
$$\sum_{n \leqslant x} \fl{x/n}^{-1/k} M_k  \(\fl{x/n}\)= x \sum_{n=1}^{\infty} \frac{M_k(n)}{n^{1+1/k}(n+1)} + O\( x^{1-1/(3k) +o_k(1)} \).$$
\end{exe}

Furthermore, if for an integer $q\ge 2$ we use $\sigma_q(n)$ to denote the sum
of $q$-ary digits of $n$, then we see that Theorem~\ref{thm:Sx}  immediately
implies:

\begin{exe}
\label{exe:digit} For any integer $q\ge 2$, we  have
$$\sum_{n \leqslant x} \sigma_q\(\fl{x/n} \)=     x \sum_{n=1}^{\infty} \frac{ \sigma_q(n)}{n(n+1)} +  O \(x^{2/3+ o(1)}\).$$
\end{exe}

\section{Proof of Theorem~\ref{thm:Sx}}

\subsection{Preliminaries}

\subsubsection{Vaaler polynomials}

For a real $z \in \R$ we denote
\begin{equation}
\label{eq:psi}
 \psi(z) = z - \fl{z} - \frac{1}{2} \mand\e(z) = \exp(2\pi i  z).
\end{equation}

We  need a  result of Vaaler~\cite{Vaal}, approximating $ \psi(z) $  via trigonometric
polynomials
which we  present in the form given by~\cite[Theorem~6.1]{Bord1}. For this, for any $0 < |t| < 1$ we put
$\varPhi(t) = \pi t (1-|t|) \cot (\pi t) + |t|$.
Note that $0 < \varPhi(t) < 1$ for $0 < |t| < 1$.

\begin{lemma}
\label{lem:Vaal}
For any real number $x\geqslant 1$ and
any positive integer $H$,
$$
   \psi(z) = - \sum_{0 < |h|  \leqslant H} \varPhi \( \frac{h}{H+1} \) \frac{\e(hz)}{2 \pi i h} + \cR_H(z),
$$
where the error term $\cR_H(z)$ satisfies
$$
   \left | \cR_H(z) \right |  \leqslant \frac{1}{2H+2} \sum_{|h|  \leqslant H} \( 1 - \frac{|h|}{H+1} \) \e(hz).
$$
\end{lemma}

\subsubsection{Initial transformation}

Let $x$ be sufficiently large and $J$ be any real number satisfying $x^{1/2} < J \leqslant x$.
Clearly
\begin{align*}
   S(x) &= \sum_{n\leqslant x/J}\varphi \(\fl{x/n} \)+\sum_{ x/J<n\leqslant x}\varphi \(\fl{x/n} \) \\
   &\leqslant \sum_{n\leqslant x/J}\varphi \(\fl{x/n} \)+\sum_{n\leqslant J}\varphi(n \( \fl{\frac{x}{n}} - \fl{\frac{x}{n+1}} \)  \\
&=\sum_{n\leqslant x/J}\varphi \(\fl{x/n} \)+x\sum_{n\leqslant J}\frac{\varphi(n)}{n(n+1)}\\
& \qquad \qquad \qquad +\sum_{n\leqslant J} \varphi(n)\(\psi\(\frac x{n+1}\)-\psi\(\frac x{n}\)\).
\end{align*}
where $\psi(z)$ is given by~\eqref{eq:psi}
Now, splittling the last sum into two ranges $n   \leqslant x^{1/2} $ and $x^{1/2}<n\leqslant J$ we obtain
\begin{equation}
\label{eq: S S03}
S(x)=S_0(x) + S_1(x)+S_2(x)+S_3(x),
\end{equation}
where
\begin{align*}
S_0(x) &= \sum_{n\leqslant x/J}\varphi \(\fl{x/n} \), \\
S_1(x) &  = x\sum_{n\leqslant J}\frac{\varphi(n)}{n(n+1)}, \\
S_2 (x)& = \sum_{n\leqslant x^{1/2}}\varphi(n)\(\psi \(\frac x{n+1}\)-\psi\(\frac x{n}\)\), \\
S_3 (x)& = \sum_{x^{1/2}<n\leqslant J}\varphi(n)\(\psi\(\frac x{n+1}\)-\psi\(\frac x{n}\)\).
\end{align*}

\subsection{The lower bound}

\subsubsection{Exponential sums twisted by the Euler totient}

We refer to~\cite[Sections~6.6.3]{Bord1} for the definition and basic properties of exponent pairs.

\begin{lemma}
\label{le1}
Let $N,N_1 \in \Z_{\geqslant 1}$ and $x > 0$  such that $N < N_1 \leqslant 2N$ and $N \leqslant x$. If $(k,\ell)$ is an exponent pair, then
$$\sum_{N < n \leqslant N_1} \varphi(n) \e \( \frac{x}{n} \) \ll x^k N^{1+ \ell-2k} \log N  + N^3 x^{-1} + N.$$
Furthermore, if $\( k, \ell \) \ne \( \frac{1}{2}, \frac{1}{2} \)$, then the factor $\log N$ may be omitted.
\end{lemma}

\begin{proof} For any arithmetic functions $f$ and $g$, $f \star g$ is the usual Dirichlet convolution product of $f$ and $g$,   defined as
$$
\left( f \star g \right) (n)  = \sum_{ d\mid n} f(d) g(n/d).
$$

Using $\varphi = \mu \star \id$, see~\cite[Equation~(4.7)]{Bord1},   we obtain
\begin{equation}
\begin{split}
\label{eq:estphi1}
   \sum_{N < n \leqslant N_1} \varphi (n)  \e \( \frac{x}{n} \) &= \sum_{N < n \leqslant N_1} \( \mu \star \id \) (n)  \e \( \frac{x}{n} \)   \\
   &= \sum_{n \leqslant N_1} \mu(n) \sum_{N/n < m \leqslant  N_1/n} m  \e \( \frac{x}{mn} \)  \\
   &= \sum_{n \leqslant N} \mu(n) \sum_ {N/n < m \leqslant  N_1/n} m  \e \( \frac{x}{mn} \) + O(N).
\end{split}
\end{equation}
For all $M , M_1 \in Z_{\geqslant 1}$ such that $M < M_1 \leqslant 2M$, if $(k,\ell)$ is an exponent pair, then by Abel summation
\begin{equation}
\begin{split}
   \sum_{M < m \leqslant M_1} m  \e \( \frac{x}{mn} \) & \ll  M \max_{M \leqslant n \leqslant M_1} \left | \sum_{M \leqslant m \leqslant n}  \e \( \frac{x}{mn} \) \right |  \\
   & \ll M \left\lbrace \( \frac{x}{Mn} \)^k M^{\ell - k} + \frac{M^2n}{x} \right\rbrace  \\
   & \ll  \( \frac{x}{n} \)^k M^{1 + \ell - 2k} + \frac{M^3n}{x}. \label{eq:estphi2}
\end{split}
\end{equation}
Inserting~\eqref{eq:estphi2} with
$$
M = \frac{N}{n} \mand M_1 = \frac{N_1}{n}
$$
in~\eqref{eq:estphi1}, we obtain
\begin{align*}
   \sum_{N < n \leqslant N_1} \varphi (n)  \e \( \frac{x}{n} \) &\ll  x^k N^{1+ \ell-2k} \sum_{n \leqslant N} \frac{1}{n^{1 + \ell-k}} + \frac{N^3}{x} \sum_{n \leqslant N} \frac{1}{n^{2}} + N \\
   &\ll x^k N^{1+ \ell-2k} (\log N)^{\alpha} + \frac{N^3}{x} + N,
\end{align*}
where $\alpha  = 1$ if $(k,\ell) = \( \frac{1}{2},\frac{1}{2} \)$ and  $\alpha = 0$ otherwise, giving the asserted result.
\end{proof}

\subsubsection{Exponent pairs and a lower bound on $S(x)$}

The desired lower bound on $S(x)$ is a particular case of the following more general result, which may have its own interest.

\begin{lemma}
\label{lem:exp-pair}
Let $x \geqslant e$ be sufficiently large and let $J$ be any real number satisfying $x^{1/2} < J \leqslant x$ and $(k,\ell)$ be an exponent pair. Then
$$
S(x)  \geqslant  \frac{x \log J}{\zeta(2)} + O(\fJ  (\log J)^2 + x),
$$
where
\begin{align*}
\fJ =   \( J^{\ell+1} x^{k+1} \)^{1/(k+2)} &+ \( J^{2(\ell+1)} x^{k} \)^{1/(k+2)} \\
& + \( J^{3k-\ell+5} x^{-k-1} \)^{1/(k+2)} + J^3/x .
\end{align*}
\end{lemma}

\begin{proof} Recalling~\eqref{eq: S S03} and using that $S_0(x) \ge 0$ we write
$$  S(x) \ge  S_1(x) + S_2(x) + S_3(x).
$$
Now
$$
S_1(x) = x \sum_{n \leqslant J} \frac{\varphi(n)}{n^2} + O(x) = \frac{x \log J}{\zeta(2)} + O(x)
$$
and obviously
$$
 S_2(x) \ll x.
$$
It remains to estimate $S_3$. Covering the interval $[x^{1/2},J]$ by $L \ll \log J$ dyadic
intervals of the form $[N,2N]$, we have
\begin{align*}
  S_3(x)& \leqslant \left | \sum_{x^{1/2} < n \leqslant J} \varphi (n) \psi \( \frac{x}{n+1} \) \right | + \left | \sum_{x^{1/2} < n \leqslant J} \varphi (n) \psi \( \frac{x}{n} \) \right | \\
   & \ll L \max_{\vartheta \in \{0,1\}} \max_{x^{1/2} < N \leqslant J}  \left | \sum_{N < n \leqslant 2N} \varphi (n) \psi \( \frac{x}{n+\vartheta} \) \right |  .
\end{align*}

Now, by Lemma~\ref{lem:Vaal}, for any integer $H \geqslant 1$,
$$
S_3(x)
  \ll L \max_{\vartheta \in \{0,1\}}\max_{x^{1/2} < N \leqslant J} \( \frac{N^2}{H} + \sum_{h \leqslant H} \frac{1}{h} \left | \sum_{N < n \leqslant 2N} \varphi (n)  \e \( \frac{hx}{n+\vartheta} \) \right | \) .
$$

Note that the function
$$
n \longmapsto \dfrac{hx}{n(n+1)}
$$
is non-increasing and bounded by $HxN^{-2}$ so that, by partial summation,
\begin{align*}
   \sum_{N < n \leqslant 2N} \varphi (n)  \e \( \frac{hx}{n+1} \) &= \sum_{N < n \leqslant 2N} \varphi (n)  \e \( \frac{hx}{n} \)  \e \( -\frac{hx}{n(n+1)} \) \\
   &\ll   \( 1 + \frac{Hx}{N^2} \) \left | \sum_{N < n \leqslant 2N} \varphi (n)  \e \( \frac{hx}{n} \)
    \right |
\end{align*}
and therefore
\begin{equation}
\label{eq:S3 prelim}
  S_3(x)   \ll  L\max_{x^{1/2} < N \leqslant J}  \( \frac{N^2}{H} +  \( 1 + \frac{Hx}{N^2} \) W\),
\end{equation}
where
$$
W =  \sum_{h \leqslant H} \frac{1}{h} \left | \sum_{N < n \leqslant 2N} \varphi (n)  \e \( \frac{hx}{n} \) \right |.
$$
The estimate of Lemma~\ref{le1} yields
\begin{align*}
W& \leqslant(Hx)^k N^{1+ \ell-2k} \log N +\frac{N^3}{x} + N \log H \\
&\leqslant  \((Hx)^k N^{1+ \ell-2k} + \frac{N^3}{x} \) \log H,
\end{align*}
where we have used the fact that $N > x^{1/2}$ implies that $N^3x^{-1} \geqslant N$.

Inserting this estimate in~\eqref{eq:S3 prelim}, we derive
$$
  S_3(x)
   \ll   L^2 \max_{x^{1/2} < N \leqslant J} \( \frac{N^2}{H} + \( 1 + \frac{Hx}{N^2} \) \( (Hx)^k N^{1+ \ell-2k} + \frac{N^3}{x} \) \).
$$
Now choose
$$
H = \left \lfloor \( N^{2k-\ell+3}x^{-k-1} \)^{1/(k+2)} \right \rfloor.
$$
Note that the condition $N > x^{1/2}$ ensures that $H \geqslant 1$. We eventually obtain
$$ S_3(x)  \ll L^2\fJ,$$
concluding the proof.
\end{proof}

\subsubsection{Concluding the proof of the lower bound}
\label{sec:concl LB}
The lower bound of Theorem~\ref{thm:Sx}  follows from Lemma~\ref{lem:exp-pair}
at once by using the  exponent pair of Bourgain~\cite[Theorem~6]{Bour}, coupled with several applications of van der Corput's $A$- and $B$-processes, see~\cite[Sections~6.4.2 and~6.6.2]{Bord1}:
$$(k,\ell) = BA^3 \( BA^2 \)^2\( \frac{13}{84} + \varepsilon, \frac{55}{84} + \varepsilon  \) = \( \frac{ {3071}}{ {7887}} + \varepsilon, \frac{ {1380}}{ {2629}} + \varepsilon  \)$$
and choosing $J=x^{ {2629/4009} - \varepsilon}$.

\subsection{The upper bound}

\subsubsection{Some explicit estimates}
The following estimate is a well-known result. For a proof, see~\cite[Lemma~2.1]{ramak}.

\begin{lemma}
\label{lem:harm}
For all $x \geqslant 1$
$$\left | \sum_{n \leqslant x} \frac{1}{n} - \log x - \gamma \right | \leqslant \frac{6}{11 x}.$$
\end{lemma}

We also need some bounds on sums involving the Euler function, w  follows from~\cite[Lemma~2.2]{tre15} and partial summation.

\begin{lemma}
\label{lem:phi}
For all $x \geqslant 1$ we have
$$\sum_{n \leqslant x} \frac{\varphi(n)}{n^2} \leqslant \frac{\log x}{\zeta(2)} + 2 + \frac{1}{\zeta(2)}.
$$
\end{lemma}

\subsubsection{Concluding the proof}
We recall the representation of $S(x)$ as~\eqref{eq: S S03}. 

First, using Lemma~\ref{lem:harm}
\begin{equation}
   S_0(x)  \leqslant  x \sum_{n \leqslant x/J} \frac{1}{n} \leqslant x \log (x/J) +   O(x). \label{eq:S0}
\end{equation}
Next, using Lemma~\ref{lem:phi} we derive
\begin{equation}
\label{eq:S1}
   S_1(x)  \leqslant   x\sum_{n\leqslant J}\frac{\varphi(n)}{n(n+1)}  \leqslant   x\sum_{n\leqslant J}\frac{\varphi(n)}{n^2}
   =  x \frac{\log J}{\zeta(2)} +O(x).
\end{equation}
Furthermore, we obviously have
\begin{equation}
 \label{eq:S2}
S_2(x)\leqslant \sum_{n\leq x^{1/2}}\varphi(n) \leqslant  \sum_{n\leq x^{1/2}}n\ll x.
\end{equation}

Now, choosing  $J=x^{ {2629/4009} - \varepsilon}$ for a sufficiently small $\varepsilon>0$ as in
Section~\ref{sec:concl LB}, we see that
\begin{equation}
 \label{eq:S3}
S_3(x)  = o(x).
\end{equation}

Now substituting the bound~\eqref{eq:S0}, \eqref{eq:S1}, ~\eqref{eq:S2} and~\eqref{eq:S3} in~\eqref{eq: S S03} implies  the asserted upper bound  Theorem~\ref{thm:Sx}.

\section{Proof of Theorem~\ref{thm:Sfx}}

\subsection{Initial transformation}
Let $T \in \( 2 \sqrt{x} ,x \right ]$ be a parameter at our disposal. Then
\begin{align*}
S_f(x) &=  \sum_{n \leqslant x} f(n) \( \left \lfloor \frac{x}{n} \right] - \left \lfloor  \frac{x}{n+1} \right \rfloor \) \\
     &= \sum_{n \leqslant T} f(n) \( \left \lfloor \frac{x}{n} \right \rfloor - \left \lfloor \frac{x}{n+1} \right \rfloor \)\\
     & \qquad \quad  + \sum_{T < n \leqslant x} f(n) \( \left \lfloor \frac{x}{n} \right \rfloor - \left \lfloor \frac{x}{n+1} \right \rfloor \) \\
     & =   \sum_{n \leqslant T}f(n)\( \frac{x}{n(n+1)} + O (1)\)  \\
     &  \qquad \quad + O \( \sum_{T < n \leqslant x} \left | f(n) \right | \( \left \lfloor \frac{x}{n} \right \rfloor - \left \lfloor \frac{x}{n+1} \right \rfloor \) \).
\end{align*}

Hence
\begin{equation}
\label{eq:R123}
S_f(x)  =  x \sum_{n=1}^{\infty} \frac{f(n)}{n(n+1)} + O\(R_1(x)+R_2(x)+R_3(x)\),
\end{equation}
 where
\begin{align*}
R_1(x)  &=   x \sum_{n> T}^{\infty} \frac{\left | f(n) \right |}{n(n+1)},\\
R_2(x)  & =   \sum_{n \leqslant T} \left | f(n) \right |,\\
R_3(x) & = \sum_{T < n \leqslant x} \left | f(n) \right | \( \left \lfloor \frac{x}{n} \right \rfloor - \left \lfloor \frac{x}{n+1} \right \rfloor \).
\end{align*}

\subsection{Bounding error terms}
Using that
$$\sum_{n \leqslant T} \left | f(n) \right | \leqslant A \sum_{n \leqslant T} \tau_k(n) \ll_k AT (\log T)^{k-1},$$
see~\cite[Section~4.8, Exercise~13]{Bord1},
and by partial summation we obtain
\begin{equation}
\label{eq:Bound R1}
R_1(x)  \leqslant   Ax\sum_{n> T}^{\infty} \frac{\tau_k(n)}{n^2}  \ll_k A T^{-1}  x (\log T)^{k-1},
\end{equation}
where the implied constant in $U \ll_k V$ (which is equivalent to  $U = O_k(V)$) may depend on $k$.

We also  have
\begin{equation}
\label{eq:Bound R2}
R_2(x)  \leqslant  \sum_{n \leqslant T} \left | f(n) \right | \leqslant A \sum_{n \leqslant T} \tau_k(n) \ll_k AT (\log T)^{k-1}.
\end{equation}

If $k=1$, then
\begin{equation}
\begin{split}
\label{eq:Bound R3 k1}
R_3(x) & \leqslant A\sum_{T < n \leqslant x} \( \left \lfloor \frac{x}{n} \right \rfloor - \left \lfloor \frac{x}{n+1} \right \rfloor \)\\
& = A\( \fl{\frac{x}{\rf{T}}} - \fl{\frac{x}{\fl{x}+1}} \)\ll AxT^{-1}.
\end{split}
\end{equation}
Choosing $T = 3 \sqrt{x}$ and substituting~\eqref{eq:Bound R1}, \eqref{eq:Bound R2}
and~\eqref{eq:Bound R3 k1} in~\eqref{eq:R123} we obtain the desired result for $k=1$.

Now we assume that $k \geqslant 2$ and estimate $R_3(x)$ in this case.
We have
\begin{equation}
\label{eq:R3 SigmaN}
R_3(x) \ll  A\max_{T < N \leqslant x}  \Sigma_N(x)  \log (x/T),
\end{equation}
where
$$\Sigma_N(x) = \sum_{N < d \leqslant 2N} \tau_k(d) \( \left \lfloor \frac{x}{d} \right \rfloor - \left \lfloor \frac{x}{d+1} \right \rfloor \).$$
Interverting the summations we obtain
\begin{equation}
\begin{split}
\label{eq:SigmaN}
   \Sigma_N(x) &= \sum_{N < d \leqslant 2N} \tau_k(d) \( \left \lfloor \frac{x}{d} \right \rfloor - \left \lfloor \frac{x}{d} - \frac{x}{d(d+1)}\right \rfloor \) \\
   & \leqslant   \sum_{N < d \leqslant 2N} \tau_k(d) \( \left \lfloor \frac{x}{d} \right \rfloor - \left \lfloor \frac{x}{d} - \frac{x - xN^{-1}}{d}\right \rfloor \) \\
   &=  \sum_{N < d \leqslant 2N} \tau_k(d) \sum_{x-xN^{-1} < md \leqslant x} 1 \\
   &=   \sum_{x-xN^{-1} < n \leqslant x} \sum_{\substack{d \mid n \\ N < d \leqslant 2N}} \tau_k(d).
\end{split}
\end{equation}
Now~\cite[Lemma~5.2]{Bord2} yields
$$\Sigma_N(x) \ll_k (\log N)^{k-1} \sum_{x-xN^{-1} < n \leqslant x} \Delta_{k+1} (n), $$
where $\Delta_r$ is the $r$th Hooley's divisor function (see~\cite{Bord2} and the references therein), and using~\cite[Lemma~5.1]{Bord2}, for an arbitrary fixed $\varepsilon>0$,   we obtain
$$\Sigma_N(x) \ll_{k,\varepsilon} \( xN^{-1} (\log x)^{\varepsilon_{k+1}(x)} + x^\varepsilon \)(\log N)^{k-1},$$
where $\varepsilon_{k+1}(x)$ is defined in~\eqref{eq:eps} (and the implied constant is now allowed to depend
on $\varepsilon$ as well).

\subsection{Concluding the proof}
Collecting the previous estimates finally gives
\begin{align*}
   \sum_{n \leqslant x} f \( \fl{x/n} \) &= x \sum_{n=1}^{\infty} \frac{f(n)}{n(n+1)} \\
   & \qquad  + O_{k,\varepsilon} \left\lbrace  AT (\log x)^{k-1} + \frac{Ax}{T} (\log x)^{k+\varepsilon_k(x)} + Ax^\varepsilon  \right\rbrace,
\end{align*}
and choosing $T = 3\sqrt{x (\log x)^{1+\varepsilon_{k+1}(x)}}$, and $\varepsilon = 1/2$,
we see that the last term never dominates, which completes the proof.

\section{Proofs of Theorems~\ref{thm:Sfxinf} and~\ref{thm:Sfx2}}
\subsection{Proof of Theorem~\ref{thm:Sfxinf}}

The proof follows closely that of Theorem~\ref{thm:Sfx} in the case $k=1$ above. It is only sufficient to note that, since $\left | f(n) \right | \ll n^{\phi - 1} (\log en)^{-A}$, the bounds~\eqref{eq:Bound R1}, \eqref{eq:Bound R2} and~\eqref{eq:Bound R3 k1} become here
\begin{align*}
 R_1(x) \ll x T^{\phi - 2}, \qquad
  R_2(x) \ll T^{\phi}, \qquad
 R_3(x) \ll x^{\phi} T^{-1} (\log T)^{-A},
\end{align*}
respectively.
 Choosing $T = x^{\phi/(\phi+1)}
(\log x)^{-A/(\phi + 1)}$ and replacing $\log T$ and with $\log x$ in the bound on $R_3(x)$ yields the asserted result.

\subsection{Proof of Theorem~\ref{thm:Sfx2}}

Again, the proof follows closely that of  Theorem~\ref{thm:Sfx} in the case $k \geqslant 2$. Firstly, note that, using the assumption~\eqref{eq:hyp-2.5} and the Cauchy--Schwarz inequality,
we obtain
$$
\sum_{n \leqslant x} \left | f(n) \right | \leqslant x^{1/2} \left( \sum_{n \leqslant x} \left | f(n) \right |^2 \right)^{1/2} \ll x^{(\alpha+1)/2}.
$$
Hence,  for any $T \in \left( 2x^{1/2},x \right]$, the bounds for $R_1(x)$ and $R_2(x)$ become $$R_1(x) \ll x T^{(\alpha - 3)/2} \mand R_2(x) \ll T^{(\alpha+1)/2},
$$
respectively.

 To estimate $R_3(x)$, assume that $T \ll x^{1-\varepsilon}$ for some fixed $\varepsilon >0$. Now
similarly to our treatment in~\eqref{eq:R3 SigmaN} and~\eqref{eq:SigmaN}, we obtain
\begin{align*}
   R_3(x) &= \sum_{T < n \leqslant x} \left | f(n) \right | \left( \left \lfloor \frac{x}{n} \right \rfloor - \left \lfloor \frac{x}{n+1} \right \rfloor \right) \\
   & \ll  \underset{T < N \leqslant x}{\mx} \left( \sum_{N < n \leqslant 2N} \left | f(n) \right | \left( \left \lfloor \frac{x}{n} \right \rfloor - \left \lfloor \frac{x}{n} - \frac{xN ^{-1}}{n}\right \rfloor \right) \right) \log x \\
   & \ll  \underset{T < N \leqslant x}{\mx} \left( \sum_{x - \frac{x}{N} < n \leqslant x} \sum_{\substack{d \mid n \\ N < d \leqslant 2N}} \left | f(d) \right | \right)  \log x.
\end{align*}
Now using the Cauchy--Schwarz inequality and~\eqref{eq:hyp-2.5}, we derive
$$\sum_{\substack{d \mid n \\ N < d \leqslant 2N}} \left | f(d) \right | \leqslant \left( \sum_{n \leqslant 2N} \left | f(n) \right |^2 \right)^{1/2} \Delta(n)^{1/2} \ll N^{\alpha/2} \Delta(n)^{1/2}$$
so that, if $N \gg x^{1-\varepsilon}$ and $0 < \alpha < 2$
\begin{align*}
R_3(x) & \ll \underset{T < N \leqslant x}{\mx} \left( N^{\alpha/2} xN^{-1}\right) (\log x)^{1+\varepsilon_2(x)/2}\\
& \ll xT^{-1+\alpha/2} (\log x)^{1+ \varepsilon_1(x)/2}
\end{align*}
and choosing $T = x^{2/3} (\log x)^{\(2+\varepsilon_1(x) \)/3}$ gives the asserted result.

\section{Numerical Results}
\label{sec:comp}

As we have mentioned, is not clear that a limit for
$$
\rho(x) = \frac{S(x)}{x \log x}
$$
exists and if it exists whether it coincides with
$$
\frac{1}{\zeta(2)} \approx 0.60\, 793.
$$
Using Maple we can calculate approximate values of  $\rho(x)$ for various values of $x$ as shown in the following table.
\begin{table}[H]
\centering
\caption{}
\begin{tabular}{|l|l|}
\hline
$x$& $\rho(x)$\\ \hline
\multirow{1}{*}{$10^6$}
& 0.5844  \\
 \hline
\multirow{1}{*}{$10^7$}
& 0.5849  \\
 \hline
 \multirow{1}{*}{$10^8$}
& 0.5896  \\
 \hline
 \multirow{1}{*}{$10^9$}
& 0.5909  \\
 \hline
 \multirow{1}{*}{$10^{10}$}
& 0.5940  \\
 \hline
\end{tabular}
\end{table}

Unlike the normal totient summation, increasing $x$ by 1 changes all the previous summands. So the ratio $\rho(x)$ can meaningfully change for a small change in $x$ as shown here.
\begin{table}[H]
\centering
\caption{}
\begin{tabular}{|l|l|}
\hline
$x$& $\rho(x)$\\ \hline
\multirow{1}{*}{$10^6$}
& 0.5844  \\
 \hline
\multirow{1}{*}{$10^6+1$}
& 0.6274  \\
 \hline
 \multirow{1}{*}{$10^6+2$}
& 0.5965  \\
 \hline
 \multirow{1}{*}{$10^6+3$}
& 0.6447  \\
 \hline
 \multirow{1}{*}{$10^{6}+4$}
& 0.6108  \\
 \hline
\end{tabular}
\end{table}

The meaningful change in the ratio $\rho(x)$ is also evident for small changes in $x$ above $10^{10}$.

\begin{table}[H]
\centering
\caption{}
\begin{tabular}{|l|l|}
\hline
$x$& $\rho(x)$\\ \hline
\multirow{1}{*}{$10^{10}$}
& 0.5940  \\
 \hline
\multirow{1}{*}{$10^{10}+1$}
& 0.6200  \\
 \hline
 \multirow{1}{*}{$10^{10}+2$}
& 0.6001  \\
 \hline
 \multirow{1}{*}{$10^{10}+3$}
& 0.6270  \\
 \hline
 \multirow{1}{*}{$10^{10}+4$}
& 0.6144  \\
 \hline
\end{tabular}
\end{table}

\section{Comments }

In the proof of the  lower bound of Theorem~\ref{thm:Sx},
one can take the slightly weaker exponent pair
$$(k,\ell) = BA^3 \( BA^2 \)^2 B\( 0,1  \) = \( \frac{ {97}}{ {251}}, \frac{ {132}}{ {251}} \)$$
and choose $J=x^{251/383} (\log x)^{- 1198/383}$ to obtain
$$\sum_{n \leqslant x} \varphi \( \fl{x/n}\) \geqslant \( \frac{ {251}}{ {383}} - \frac{{1198} \log \log x}{383 \log x} \) \frac{x \log x}{\zeta(2)} + O ( x )$$
and similarly for the upper bound.  
Note  that 
$$\frac{251}{383 \zeta(2)} \approx 0.39\,841
  \mand \frac{2629}{4009 \zeta(2)} \approx 0.39\,866
$$
while in the upper bound we have
$$
 \frac{251}{383}\cdot\frac1{\zeta(2)} + \frac{132}{383} \approx 0.74\,305 \quad \text{and}\quad
 \frac{2629}{4009}\cdot\frac1{\zeta(2)} + \frac{1380}{4009} \approx 0.74\,289.
$$ 
We also note that one can obtain an asymptotically weaker but fully explicit form of the upper
bound on $S(x)$, which does not rely on exponent pairs. Namely, taking $J= x^{1/2}$ so that the sum $S_3(x)$ 
in~\eqref{eq: S S03} becomes trivial, and using the explicit bound
$$
 \sum_{n \leqslant x} \varphi(n) \leqslant \frac{x^2}{2\zeta(2)} + x\log x + 2x, 
$$
which can be   derived by combining~\cite[Lemma~3.1]{tre14} with~\cite[Lemma~2.2]{tre15}, 
one obtains
$$
S(x) \leqslant  \frac{1}{2} \left( 1 + \frac{1}{\zeta(2)} \right)  x \log x + 4x + \frac{\sqrt{x} \log x}{4} + \sqrt{x},
$$
for any $x \ge 3$. For comparison,
$$
 \frac{1}{2} \left( 1 + \frac{1}{\zeta(2)} \right)  \approx 0.80\,396.
$$

 \section*{Acknowledgement}

During the preparation of this work, L. Dai   was supported by National Natural Science Foundation of China (Grant~11571174) and Qing Lan Project of Nanjing Normal University, H.~Pan  by National Natural Science Foundation of China (Grant~11671197) and I.~E.~Shparlinski by  
the  Australian Research Council  (Grants DP170100786).

\end{document}